\documentclass{elsarticle}

\usepackage[english]{babel}

\usepackage{tikz,pgfplots,tikz-network}

\usepackage[letterpaper,top=2cm,bottom=2cm,left=3cm,right=3cm,marginparwidth=1.75cm]{geometry}

\usepackage{amsmath,amsthm,amssymb}
\usepackage{graphicx}
\usepackage[colorlinks=true, allcolors=blue]{hyperref}
\usepackage{comment}
\usepackage{algorithm,algpseudocode}

\newtheorem{theorem}{Theorem}
\newtheorem{proposition}{Proposition}
\newtheorem{lemma}{Lemma}
\newtheorem*{theorem*}{Proposition*}

\DeclareMathOperator{\nullity}{Null}

\journal{Linear Algebra and its Applications}

\begin{document}

\begin{frontmatter}

\title{A Block-Shifted Cyclic Reduction Algorithm for Solving a Class of Quadratic Matrix Equations}
\author{Xu Li}
\ead{lixu@lut.edu.cn}
\affiliation{organization={Department of Applied Mathematics, Lanzhou University of Technology},
            city={Lanzhou},
            postcode={730050},
            country={China}}
            
            \author{Beatrice Meini}
            \ead{beatrice.meini@unipi.it}
\affiliation{organization={Dipartimento di Matematica, Università di Pisa},
            city={Pisa},
            postcode={56127},
            country={Italy}}

\begin{abstract}
The cyclic reduction (CR) algorithm is an efficient method for solving quadratic matrix equations that arise in quasi-birth–death (QBD) stochastic processes. However, its convergence is not guaranteed when the associated matrix polynomial has more than one eigenvalue on the unit circle. To address this limitation, we introduce a novel iteration method, referred to as the Block-Shifted CR  algorithm, that improves the CR algorithm by utilizing singular value decomposition (SVD) and block shift-and-deflate techniques. This new approach extends the applicability of existing solvers to a broader class of quadratic matrix equations. Numerical experiments demonstrate the effectiveness and robustness of the proposed method.
\end{abstract}


\end{frontmatter}

\section{Introduction}

Let $A_i\in \mathbb{R}^{m\times m}$, $i=0,1,2$, be given matrices, and consider the quadratic matrix equation (QME)
\[
A_0 + A_1 X + A_2 X^2 = 0,
\]
where the unknown $X$ is an $m\times m$ matrix.
This unilateral QME represents an important class of nonlinear matrix problems, whose study dates back to the pioneering work of Sylvester \cite{sylvester1884}.
They arise in a broad range of applications, including quasi-birth-death (QBD) processes in stochastic models and queueing theory \cite{latouche1999,BLM05,Meini:survey}, algebraic Riccati equations (AREs) in control theory \cite{lancaster1995}, quadratic eigenvalue problems (QEPs) in vibration analysis and structural mechanics \cite{Lambda-matrices,Dennis1976,QEP}, and noisy Wiener-Hopf problems for Markov chains \cite{Wiener-Hopf1,Wiener-Hopf2}. 

A commonly used numerically stable approach for solving QMEs is to linearize them into associated generalized eigenvalue problems, which can then be efficiently solved using the QZ algorithm to compute all the solutions \cite{Higham2000,QEP,gvl}. However, the computational cost limits its practicality for large-scale problems.
Iterative methods, such as Newton’s method \cite{Higham2000,higham2001}, the Bernoulli iteration (BI) and its variants \cite{Higham2000,bai2007}, the fixed-point iterations \cite{Bai1997,Meini1997,Guo1999,Favati1999,Higham2000,BLM05,Bini2022}, offer greater efficiency but may converge slowly or stagnate in nearly null-recurrent QBDs.
An alternate to classical algorithms is
the cyclic reduction (CR) algorithm. Originally developed by Golub and Hockney for block tridiagonal systems \cite{buzbee1970,Hockney}, the CR algorithm was later extended to solve nonlinear matrix equations based on a functional representation \cite{BLM05,meini2006}. The central idea is to transform the nonlinear matrix equation into an equivalent semi-infinite, block Toeplitz, block tridiagonal or Hessenberg linear system and apply CR to solve this system. The comprehensive survey  \cite{BM:CRsurvey} provides a detailed account of the development and applications of CR.

The CR algorithm achieves high efficiency and robustness for solving QMEs, with its convergence behaviour closely related to the spectral properties of the associated matrix polynomial. When the eigenvalues of the associated matrix polynomial can be split into two sets---one inside and one outside the unit disk--the CR algorithm converges quadratically to the solution with minimal spectral radius, making it the preferred choice for many problems in queueing theory, particularly  for positive recurrent and transient QBDs \cite{BLM05,smcsolver}. As a variant of CR for QBD processes, the Logarithmic-Reduction (LR) algorithm \cite{LR} has the same convergence rate in view of the relationship between CR and LR \cite{CR-LR}. CR is also closely related to the Structured Doubling Algorithm \cite{huang2018,chiang2009}.

For null recurrent QBDs, the LR algorithm is proven to converge linearly with rate $1/2$ under two additional  assumptions \cite{guo2002}, thus the CR algorithm also exhibits linear convergence with the same rate. Here, the first assumption in \cite{guo2002} is that $\lambda=1$ is a simple eigenvalue of the sought solution and there are no other eigenvalues of the solution on the unit circle; the second assumption is more technical. To improve the convergence rate, a shifted CR (S-CR) algorithm \cite{HMR} was proposed for QBDs, which shifts the known eigenvalues away from the unit circle to restore quadratic convergence. Later, Guo \cite{guo2003} demonstrated that the S-CR algorithm still achieves quadratic convergence for null-recurrent QBDs. These methods form the basis  shift-based strategies for structured matrix equations \cite{bea:shift,BM:blockshift,Bini2017}. 

Nevertheless, for null-recurrent QBDs, when the associated polynomial has more than one eigenvalue on the unit circle, all of which are unknown, the applicability of existing CR-type solvers remains limited. This situation occurs when a suitable cyclicity index $\ell$ is greater than 1 (see \cite{ght96}). To address this limitation, we develop an enhanced CR algorithm, referred to as the Block-Shifted CR (\texttt{BS-CR}) algorithm, that combines singular value decomposition (SVD) with block shift-and-deflate techniques. In this framework, the CR algorithm, together with SVD, is employed to identify the invariant subspace of the sought solution associated with eigenvalues inside the open unit disk, while the block shift-and-deflate procedure separates this subspace from the subspace corresponding to eigenvalues on the unit circle. Since typically $\ell$ is much smaller then $m$, the resulting low-dimensional quadratic matrix equation, whose matrix polynomial maintains all eigenvalues of modulus one, is efficiently solved using the QZ algorithm, and the sought solution is subsequently reconstructed from the quantities obtained in this process. Consequently, this hybrid approach extends the applicability of CR-type solvers to a broader class of quadratic matrix equations and ensures reliable convergence even in challenging null-recurrent cases. Moreover, from the numerical experiments, the proposed approach is much more accurate than classical CR.

The rest of this paper is organized as follows: Section \ref{sec:Preliminaries} presents preliminaries and assumptions, including the classical CR algorithm and its convergence theory, and the block shift-and-deflate technique. The deflation of eigenvalues not lying on the unit circle is established in Section \ref{sec:defl}. The convergence properties of CR in the case of more than one eigenvalue on the unit circle are analyzed in Section~\ref{sec:convcr}. Section~\ref{CR and block shift-and-deflate} develops the \texttt{BS-CR} algorithm. Numerical experiments in Section~\ref{Numerical results} demonstrate the effectiveness of our algorithm, and conclusions are drawn in Section~\ref{Conclusions}.

\section{Preliminaries and assumptions}\label{sec:Preliminaries}

Throughout this paper, we use the following notation. Given a square matrix $A$, we denote by $\rho(A)$ its spectral radius and by $\sigma(A)$ the set of its eigenvalues. The identity matrix of size $n$ is denoted by $I_n$, and the zero matrix and the all-ones vector are denoted by $0$ and $\mathbf{1}$, respectively.

Given the $ m\times m$ real matrices  $A_i$, for $i=0,1,2$,  define  the quadratic matrix polynomial 
\begin{equation}\label{eq:qmp}
A(z)=A_0+z A_1+z^2 A_2
\end{equation}
and the scalar
polynomial
\begin{equation}\label{eq:scalpol}
a(z)=\det A(z).    
\end{equation}
Assume $a(z)$ not identically zero 
and denote its roots as $\lambda_i$, $i=1,\ldots,2m$, adding $s$ zeros at infinity if $a(z)$ has degree $2m-s$.
Without loss of generality, assume that the roots are ordered according to their modulus
as
\[
|  \lambda_1 | \le |\lambda_2 | \le \cdots \le | \lambda_{2m} |.
\]
The roots of $a(z)$ are called the eigenvalues of the matrix polynomial $A(z)$.
We associate with \eqref{eq:qmp} the QME
\begin{equation}\label{eq:qme}
A_0 + A_1 X + A_2 X^2 = 0,
\end{equation}
together with its reversed form
\begin{equation}\label{eq:rqme}
X^2 A_0 + X A_1 + A_2 = 0.
\end{equation}
The interest will be the computation of the solutions $G$ and $R$ of the QME \eqref{eq:qme} and \eqref{eq:rqme}, respectively, with minimal spectral radius, which are typically the solution of interest in the applications.

\subsection{Cyclic Reduction (CR)}
CR is an algorithm that can be efficiently applied to solve the QMEs \eqref{eq:qme} and \eqref{eq:rqme} in the case where $|\lambda_m|<|\lambda_{m+1}|$. In particular, assuming that \eqref{eq:qme} has a solution $G$ with $\rho(G)=|\lambda_m|$, and \eqref{eq:rqme} has a solution $R$ with $\rho(R)=1/|\lambda_{m+1}|$, CR generates two sequences of approximations quadratically convergent to $G$ and $R$, respectively.
For an introduction to CR and to its main properties, we refer the reader to \cite{BM:CRsurvey,BLM05}.

More specifically, CR consists of generating the sequences of matrices
\begin{equation}\label{eq:crgen}
    \begin{aligned}
& A_{0}^{(k+1)}=A_{0}^{(k)}\left(A_{1}^{(k)}\right)^{-1} A_{0}^{(k)}, \\
& A_{1}^{(k+1)}=A_{1}^{(k)}-A_{0}^{(k)}\left(A_{1}^{(k)}\right)^{-1} A_{2}^{(k)}-A_{2}^{(k)}\left(A_{1}^{(k)}\right)^{-1} A_{0}^{(k)}, \\
& A_{2}^{(k+1)}=A_{2}^{(k)}\left(A_{1}^{(k)}\right)^{-1} A_{2}^{(k)},\\
& \widehat{A}_{1}^{(k+1)}=\widehat{A}_{1}^{(k)}-A_{2}^{(k)}\left(A_{1}^{(k)}\right)^{-1} A_{0}^{(k)},
\end{aligned}
\end{equation}
assuming that $\det A_1^{(k)}\ne 0$, for $k=0,1, \ldots$, and $\widehat{A}_{1}^{(0)}=A_{1}, A_i^{(0)}=A_i, i=0,1,2$.

From the properties of CR, if $G$ and $R$ solves the QME \eqref{eq:qme} and \eqref{eq:rqme}, respectively, then 
\begin{equation}\label{eq:crqme}
\begin{split}
&     A_0^{(k)}+A_1^{(k)}G^{2^k}+A_2^{(k)} G^{2^{k+1}}=0,\\
& R^{2^{k+1}} A_0^{(k)}+R^{2^k}A_1^{(k)}+ A_2^{(k)}=0,\\
& A_{0}+\widehat{A}_{1}^{(k)} G + A_{2}^{(k)} G^{2^k+1}=0,\\
& A_{2}+ R \widehat{A}_{1}^{(k)}  + R^{2^k+1} A_{0}^{(k)} =0.
\end{split}
\end{equation}
In particular, if $\operatorname{det}\left(\widehat{A}_{1}^{(k)}\right) \neq 0$, from the latter equations we may recover
\[
\begin{split}
&G=-\left(\widehat{A}_{1}^{(k)}\right)^{-1}\left(A_{0}+A_{2}^{(k)} G^{2^k+1}\right),\\
&    R=-\left(A_{2}+R^{2^k+1}A_{0}^{(k)} \right) \left(\widehat{A}_{1}^{(k)}\right)^{-1}.
\end{split}  
\]

Concerning convergence properties, we recall the following result, which follows from Theorems 9 and 10 of \cite{BM:CRsurvey} (see also Theorem 7.6 in \cite[Section 7.3]{BLM05}), and expresses the quadratic convergence:

\begin{theorem}\label{thm:crconv_orig}
Assume that $|\lambda_m|<1<|\lambda_{m+1}|$ and the QME \eqref{eq:qme} and $A_0 X^2 + A_1 X+ A_2=0$ have solution $G$ and $S$, respectively, with $\rho(G)<1$ and $\rho(S)<1$. If CR can be carried out without a breakdown, then
\[
\begin{split}
    & \limsup_{k\to\infty}{\| A_0^{(k)}\|^{1/2^k}}\le \rho(G),\\
  &  \limsup_{k\to\infty}{\| A_2^{(k)}\|^{1/2^k}}\le \rho(S),\\
   & \limsup_{k\to\infty}{\| G-G_k\|^{1/2^k}}\le \rho(G)\rho(S),\\
   & \limsup_{k\to\infty}{\| R-R_k\|^{1/2^k}}\le \rho(G)\rho(S),
\end{split}
\]
where $G_k=-\left(\widehat{A}_{1}^{(k)}\right)^{-1}A_{0}$ and
$R_k=-A_2 \left(\widehat{A}_{1}^{(k)}\right)^{-1}$.
\end{theorem}

In the case where $\lambda_m=\lambda_{m+1}=1$, $|\lambda_{m-1}|<1$ and $|\lambda_{m+2}|>1$, under mild conditions, the convergence of CR turns to linear \cite{guo2002}.
The convergence is restored to quadratic by applying a shifting technique to remove the eigenvalue 1 from the eigenvalues of $A(z)$ (\cite{HMR}, \cite[Section 9.2]{BLM05}).

\subsection{Assumptions}\label{sec:ass}
Throughout the paper, we assume that, for a given $1\le \ell < m$, the polynomial
$a(z)$ has $\ell$ distinct roots $\mu_1,\ldots,\mu_\ell$ of modulus $1$, each with multiplicity 2, so that 
\begin{equation}\label{eq:roots_separation}
|  \lambda_{m-\ell} | < 1= |\mu_1 | = \cdots = |\mu_\ell|
<  | \lambda_{m+\ell+1} |.
\end{equation}
Assume that there exist unique solutions $G\in\mathbb{R}^{m\times m}$ and $R\in\mathbb{R}^{m\times m}$ to the matrix equations \eqref{eq:qme} and \eqref{eq:rqme}, respectively, 
 where $G$ has eigenvalues $\lambda_1,\ldots,\lambda_{m-\ell},\mu_1,\ldots,\mu_\ell$ and $R$ has eigenvalues $1/\lambda_{m+\ell+1},\ldots,1/\lambda_{2m}$, $1/\mu_1,\ldots,1/\mu_\ell$, with the convention that $1/\infty=0$ and $1/0=\infty$.
 The uniqueness of $G$ and $R$ implies that $A(\mu_i)$ has a kernel of dimension 1, for $i=1,\ldots,\ell$. The existence of such solutions $G$ and $R$ implies that $A(z)$
 can be factorized as
\[
A(z)=(I_m-zR)H(zI_m-G),
\]
where $H$ is a suitable nonsingular matrix; moreover, the matrix $A_1+A_2G $ is nosingular, and the matrices $G$ and $R$ are related by the equation (see \cite[Section 3.3]{BLM05}):
\begin{equation}\label{eq:GRrelation}
    R = -A_2(A_1 + A_2 G )^{-1}. 
\end{equation}

These assumptions on the roots and on the existence and uniqueness of $G$ and $R$ are naturally satisfied in null recurrent QBD stochastic processes, where
 $A_0=-E_0$, $A_1=I_m-E_1$, $A_2=-E_2$ and $E_i\in\mathbb{R}^{m\times m}$, for $i=0,1,2$, have nonnegative entries and such that $(E_0+E_1+E_2)\mathbf{1}=\mathbf{1}$,
and moreover the number $\ell$ is
\[
\ell = \max\{ k :
z^{-m/k} a(z^{1/k} ) \text{ is a (single valued) function in } z\in\mathbb{C},~|z| \le  1\}.
\]
In this case, $\mu_1,\ldots,\mu_\ell$ are the $\ell$-th roots of $1$. Moreover, $G$ and $R$ have nonnegative entries and, among all the possible solutions of \eqref{eq:qme} and \eqref{eq:rqme}, respectively, are the minimal ones according to the component-wise ordering (for details on these topics we refer the reader to \cite{latouche1999}, \cite[Section 4.4]{BLM05}, \cite{ght96}).

\subsection{Block shift-and-deflate technique}\label{sec:bsd}

Denote by $\operatorname{rev}(A(z))$ the reversed quadratic matrix polynomial defined by
\[
\operatorname{rev}(A(z)) = z^{2}A\!\left(z^{-1}\right) = A_{2} + zA_{1} + z^{2}A_{0},
\]
which is obtained by reversing the order of the coefficient matrices of $A(z)$. 
It is straightforward to verify that, for any $\lambda \in \mathbb{C}\setminus\{0\}$, $\lambda$ is an eigenvalue of $A(z)$ if and only if $\lambda^{-1}$ is an eigenvalue of $\operatorname{rev}(A(z))$. Under the convention that $1/\infty = 0$ and $1/0 = \infty$, the eigenvalue $\infty$ of $A(z)$ corresponds to the eigenvalue $0$ of $\operatorname{rev}(A(z))$, and vice-versa.

The following result from \cite{BM:blockshift} describes how a set of eigenvalues of $A(z)$ can be shifted to $0$ via a right transformation.

\begin{lemma}\label{Lem:block_right_shift_0}\cite[Section 3]{BM:blockshift}
Let $1\le q<m$, 
$S_2\in \mathbb{C}^{q\times q}$,
and $V_2\in \mathbb{C}^{m\times q}$ be full rank matrices such that
\[
A_0 V_2 + A_1 V_2 S_2 + A_2 V_2 S_2^2=0.
\]
Let $Y\in\mathbb{C}^{q\times m}$ such that $Y V_2=I_q$.
Then the function
\[
A^{(R)}(z)= A(z)\left(I_m+V_2 (zI_q-S_2)^{-1}S_2Y \right)    
\]
is a quadratic matrix polynomial, having $q$ eigenvalues at $0$,  and matrix coefficients
\[
\begin{cases}
A^{(R)}_0=A_0-A_0Q, \\
A^{(R)}_1=A_1+A_2V_2S_2Y, \\
A^{(R)}_2=A_2,
\end{cases}    
\]
where  $Q=V_2 Y$.
Moreover, the eigenvalues of $A^{(R)}(z)$ are those of $A(z)$ except for the eigenvalues of $S_2$, which are moved to $0$. In particular, 
$A^{(R)}_0 V_2=0.$       
\end{lemma}

An analogous result holds for the left eigenvectors. Although not explicitly stated in \cite{BM:blockshift}, the following lemma can be obtained by arguments similar to those used for the right transformation, allowing a set of eigenvalues to be shifted to $0$.

\begin{lemma}\label{Lem:block_left_shift_0}
Let $1\le q<m$, 
$S_1\in \mathbb{C}^{q\times q}$,
and $U_1\in \mathbb{C}^{q\times m}$ be full rank matrices such that
\[
U_1 A_0  + S_1 U_1 A_1  + S_1^2U_1 A_2=0.
\]
Let $X\in \mathbb{C}^{m\times q}$ such that $ U_1 X=I_q$.
Then the function
\[
A^{(L)}(z)= \left(I_m+XS_1(zI_q-S_1)^{-1} U_1 \right)A(z)    
\]
is a quadratic matrix polynomial, having $q$ eigenvalues at $0$,  and matrix coefficients
\[
\begin{cases}
A^{(L)}_0=A_0-PA_0, \\
A^{(L)}_1=A_1+XS_1U_1A_2, \\
A^{(L)}_2=A_2,
\end{cases}   
\]
where  $P=X U_1$.
Moreover, the eigenvalues of $A^{(L)}(z)$ are those of $A(z)$ except for the eigenvalues of $S_1$, which are moved to $0$. In particular, 
$U_1 A^{(L)}_0=0.$   
\end{lemma}

The next lemma addresses the case where certain eigenvalues of $\operatorname{rev}(A(z))$ are shifted to $0$, which is mathematically equivalent to shifting the corresponding eigenvalues of $A(z)$ to infinity.

\begin{lemma}\label{Lem:block_left_shift_infinity_new}
Let $1\le q<m$, 
$S_1\in \mathbb{C}^{q\times q}$ 
and let $U_1\in \mathbb{C}^{q\times m}$ be a full rank matrix such that
\[
S_1^2 U_1 A_0  + S_1 U_1 A_1  + U_1 A_2=0.
\]
Let $X\in \mathbb{C}^{m\times q}$ such that $ U_1 X=I_q$.
Then the function
\[
\widehat{A}^{(L)}(z)= \left(I_m+zXS_1 (I_q-z S_1)^{-1} U_1 \right)A(z)    
\]
is a quadratic matrix polynomial, having $q$ eigenvalues at infinity,  and matrix coefficients
\[
\begin{cases}
\widehat{A}^{(L)}_0=A_0, \\
\widehat{A}^{(L)}_1=A_1+XS_1 U_1A_0, \\
\widehat{A}^{(L)}_2=A_2-PA_2.
\end{cases}   
\]
where  $P=X U_1$.
Moreover, the eigenvalues of $\operatorname{rev}(\widehat{A}^{(L)}(z))$ are those of $\operatorname{rev}(A(z))$ except for the eigenvalues of $S_1$, which are moved to $0$. In particular, 
$U_1 \widehat{A}^{(L)}_2=0.$   
\end{lemma}
\begin{proof}
Applying Lemma \ref{Lem:block_left_shift_0} to the reversed matrix polynomial $\operatorname{rev}(A(z))$, by shifting the eigenvalues of $S_1$ to $0$, the matrix polynomial $\widehat{A}^{(L)}(z)$ is the reversed polynomial of \[\left(I_m+XS_1(zI_q-S_1)^{-1} U_1 \right)\operatorname{rev}(A(z)).\]
\end{proof}

The following theorem generalizes the previous results by providing a unified transformation that simultaneously shifts two sets of eigenvalues, associated with right and left eigenvectors, to $0$ and infinity, respectively.

\begin{theorem}
\label{th3.2}
Let $1\le q<m$, 
$S_1\in \mathbb{C}^{q\times q}$, $S_2\in \mathbb{C}^{q\times q}$,
and $U_1\in \mathbb{C}^{q\times m}$, 
$V_2\in \mathbb{C}^{m\times q}$ be full rank matrices such that
\[
S_1^2 U_1 A_0  + S_1 U_1 A_1  +  U_1 A_2=0,~~
A_0 V_2 + A_1 V_2 S_2 + A_2 V_2 S_2^2=0.
\]
Let $X\in \mathbb{C}^{m\times q}$, $Y \in\mathbb{C}^{q\times m}$ such that $ U_1 X=Y V_2=I_q$.
 Then the function
\[
\widetilde{A}(z)= \left(I_m+zXS_1 (I_q-z S_1)^{-1} U_1 \right)A(z)\left(I_m+V_2 (zI_q-S_2)^{-1}S_2Y \right)    
\]
is a quadratic matrix polynomial, having $q$ eigenvalues at $0$, $q$ eigenvalues at infinity,  and matrix coefficients
\[
\begin{cases}
\widetilde{A}_0=A_0-A_0Q, \\
\widetilde{A}_1=A_1+A_2V_2S_2Y +XS_1U_1 {\widetilde{A}_0}, \\
\widetilde{A}_2=A_2-PA_2.
\end{cases}    
\]
where  $P=X U_1$, $Q=V_2 Y$.
Moreover, among the eigenvalues of $A(z)$, those corresponding to $S_2$ are moved to $0$. Similarly, among the eigenvalues of $\operatorname{rev}(A(z))$, those corresponding to $S_1$ are moved to $0$. In particular, 
$\widetilde{A}_0 V_2=0$
and 
$U_1 \widetilde{A}_2=0.$   
\end{theorem}
\begin{proof}
The conclusions follow immediately by successively applying the right shift transformation from Lemma \ref{Lem:block_right_shift_0} and the left shift transformation from Lemma \ref{Lem:block_left_shift_infinity_new} to the matrix polynomial $A(z)$.
\end{proof}

\begin{theorem}\label{th5}
Assume that the $m \times m$ matrix polynomial $A(z)=A_0+z A_1+z^2 A_2$ can be factorized as
$$
A(z)=(I_m-zR)H(zI_m-G)
$$
where $R, G$ and $H$ are $m \times m$ matrices. Let $S_1, S_2, U_1, V_2$, and $\widetilde{A}(z)$ be as in Theorem~\ref{th3.2}. Moreover, assume that $U_1 R=S_1 U_1$ and $G V_2= V_2S_2$. Then the matrix polynomial $\widetilde{A}(z)$ can be factorized as 
$$\widetilde{A}(z)=(I_m-z\widetilde{R})H(zI_m-\widetilde{G}),$$ 
where $\widetilde{G}=G-V_2S_2 Y$ and $\widetilde{R}=R-XS_1 U_1$.
\end{theorem}

From the above theorem, it follows that the matrices $\widetilde G$ and $\widetilde{R}$ are the solutions of the QME $\widetilde A_0 + \widetilde A_1 X+ \widetilde A_2 X^2=0$ and $X^2 \widetilde A_0 + X \widetilde A_1 + \widetilde A_2 =0$, respectively, with eigenvalues the eigenvalues of $G$ and $R$, except for the ones coinciding with the eigenvalues of $S_2$ and $S_1$, respectively, which are replaced by 0.

\section{Deflation of the eigenvalues not lying on the unit circle}\label{sec:defl}
Here we assume that the conditions stated in
Section~\ref{sec:ass} hold.
We show that if the right and left invariant subspace of $G$ and $R$, respectively, corresponding to the eigenvalues inside the unit circle, is known, then the computation of $G$ and $R$ can be reduced to solving a QME whose solution has eigenvalues lying on the unit circle only.

Let $W_{G,1}\in\mathbb{R}^{m\times (m-\ell)}$ be a full rank matrix such that
\begin{equation}\label{G_invariant}
G W_{G,1} = W_{G,1}\Lambda_{G,1},   
\end{equation}
where $\Lambda_{G,1}\in\mathbb{R}^{(m-\ell)\times (m-\ell)}$ and $\sigma(\Lambda_{G,1})=\{\lambda_1,\ldots,\lambda_{m-\ell}\}$.
Similarly, let $T_{R,1}\in\mathbb{R}^{(m-\ell)\times m}$ be a full rank matrix such that
\begin{equation}\label{R_invariant}
T_{R,1} R =  \Lambda_{R,1} T_{R,1},
\end{equation}
where $\Lambda_{R,1}\in\mathbb{R}^{(m-\ell)\times (m-\ell)}$ and $\sigma(\Lambda_{R,1})=\{1/\lambda_{m+\ell+1},\ldots,1/\lambda_{2m}\}$.
Without loss of generality, the matrices $W_{G,1}$ and $T_{R,1}^\top$ can be assumed to have orthonormal columns, so that $W_{G,1}^\top W_{G,1}=T_{R,1} T_{R,1}^\top=I_{m-\ell}$.
Therefore, the columns of $W_{G,1}$  span the right invariant subspace of $G$ corresponding to the eigenvalues of modulus less than 1, while the rows of $T_{R,1}$  span the left invariant subspace of $R$ corresponding to the eigenvalues of modulus less than 1.

In particular, since $G$ and $R$ solve \eqref{eq:qme} and \eqref{eq:rqme}, respectively, we have the following:
\begin{equation}\label{10}
A_0 W_{G,1}+A_1 W_{G,1}\Lambda_{G,1}+
A_2 W_{G,1}\Lambda_{G,1}^2=0,~~~
\Lambda_{R,1}^2 T_{R,1}A_0+
\Lambda_{R,1} T_{R,1}A_1 +
 T_{R,1}A_2 = 0.    
\end{equation}

By applying Theorem~\ref{th3.2} with
$S_1={\Lambda}_{R,1}$, $U_1=X^\top=T_{R,1}$,  $S_2={\Lambda}_{G,1}$, $V_2=Y^\top=W_{G,1}$, we find
that the function
\[
\widetilde{A}(z)= \left(I_m+z T_{R,1}^\top {\Lambda}_{R,1}(I_q-z{\Lambda}_{R,1})^{-1} T_{R,1}\right)A(z)\left(I_m+W_{G,1}(zI_q-{\Lambda}_{G,1})^{-1}{\Lambda}_{G,1}W_{G,1}^\top\right)    
\]
is a quadratic matrix polynomial, having $m-\ell$ eigenvalues equal to 0, 
$m-\ell$ eigenvalues equal to infinity, and the remaining eigenvalues equal to $\mu_1,\ldots,\mu_\ell$, each with multiplicity 2. The matrix coefficients are
\[
\begin{cases}
\widetilde{A}_0=A_0-A_0W_{G,1}W_{G,1}^\top, \\
\widetilde{A}_1=A_1+A_2W_{G,1}{\Lambda}_{G,1}W_{G,1}^\top+T_{R,1}^\top{\Lambda}_{R,1}T_{R,1}\widetilde{A}_0, \\
\widetilde{A}_2=A_2-T_{R,1}^\top T_{R,1} A_2.
\end{cases}    
\]
Moreover, 
\[
\widetilde{A}_0W_{G,1}=0,~~T_{R,1}\widetilde{A}_2=0.
\]

In addition, it follows that the matrices $\widetilde{G}=G-W_{G,1}{\Lambda}_{G,1} W_{G,1}^\top$ and $\widetilde{R}=R-T_{R,1}^\top {\Lambda}_{R,1}T_{R,1}$ are the solutions of the quadratic matrix equations 
\[
\widetilde{A}_{0}+\widetilde{A}_{1} X+\widetilde{A}_{2} X^2=0,   
\]
and
\[
X^2\widetilde{A}_{0}+X\widetilde{A}_{1} +\widetilde{A}_{2} =0,   
\]
respectively, with
$\sigma(\widetilde{G})=\{0,\ldots,0,\mu_1,\ldots,\mu_\ell\}$, $\sigma(\widetilde{R})=\{0,\ldots,0,1/\mu_1,\ldots,1/\mu_\ell\}$, and $\widetilde A(z)$ can be factorized as
\begin{equation}\label{eq:facat}
\widetilde{A}(z)=(I_m-z\widetilde{R}) \widetilde H(zI_m-\widetilde{G})
\end{equation}
for a suitable nonsingular matrix $\widetilde H$.

Now, define the $m\times m$ matrices 
\[
W_G=\begin{bmatrix}
 W_{G,2} & W_{G,1}
\end{bmatrix} \quad {\text{and}} \quad T_R=\begin{bmatrix}
 T_{R,2} \\ T_{R,1}
\end{bmatrix},
\]
where $W_{G,2}\in\mathbb{R}^{m\times \ell}$ and $T_{R,2}\in\mathbb{R}^{\ell \times m}$ are such that $W_G^\top W_G=T_R T_R^\top=I_m$.

Define the matrix polynomial
$$\bar{A}(z)=T_R \widetilde{A}(z) W_G.$$ 
By multiplying \eqref{eq:facat} on the left by $T_R$ and on the right by $W_G$, we find that $\bar{A}(z)$ can be factorized as 
$$\bar{A}(z)=(I_m-z\bar{R})\bar{H}(zI_m-\bar{G}),$$ 
where 
\begin{equation}\label{eq:gbar}
\bar{G}=W_G^\top\widetilde{G}W_G, ~~\bar{R}=T_R\widetilde{R}T_R^\top,
\end{equation}
and $\bar{H}=T_R \widetilde{H} W_G$.
It follows that $\bar{G}$ and $\bar{R}$ solve 
the quadratic matrix equations
\[
\bar{A}_{0}+\bar{A}_{1} X+\bar{A}_{2} X^2=0,
\]
\[
X^2 \bar{A}_{0}+X \bar{A}_{1}+\bar{A}_{2}=0,
\]
respectively, with $\bar{A}_{i}=T_R\widetilde{A}_{i}W_G$, for $i=0,1,2.$

\begin{proposition}
The matrices $\bar{G}$ and $\bar{R}$, defined in \eqref{eq:gbar}, have the structure
\begin{equation}\label{17}
\bar{G}=\left[\begin{array}{ll}
\bar{G}_{11}  & 0 \\
\bar{G}_{21} & 0
\end{array}\right],~~
\bar{G}_{11} =  W_{G,2}^\top G W_{G,2},~~
\bar{G}_{21} = W_{G,1}^\top G W_{G,2},
\end{equation}
and
\begin{equation}\label{18}
\bar{R}=\left[\begin{array}{ll}
\bar{R}_{11} & \bar{R}_{12} \\
0 & 0
\end{array}\right],~~
\bar{R}_{11} = T_{R,2} R T_{R,2}^\top,~~
\bar{R}_{12} = T_{R,2} R T_{R,1}^\top.
\end{equation}
\end{proposition}
\begin{proof}
A direct computation shows that
\begin{equation*}
    	\begin{split}
\bar{G}&=W_G^\top\widetilde{G} W_G
=W_G^\top\left(G-W_{G,1}{\Lambda}_{G,1} 
W_{G,1}^\top \right)[W_{G,2} ~ 
W_{G,1}]\\
          &=W_G^\top[G W_{G,2} ~ G W_{G,1}-W_{G,1}{\Lambda}_{G,1}]
            =\left[\begin{array}{ll}
W_{G,2}^\top \\
W_{G,1}^\top
\end{array}\right][G W_{G,2} ~ 0]\\
            &=\left[\begin{array}{ll}
W_{G,2}^\top G W_{G,2}  & 0 \\
W_{G,1}^\top G W_{G,2} & 0
\end{array}\right].
    	\end{split}
\end{equation*}
Similarly, we proceed for $\bar{R}$.
\end{proof}

By direct inspection, we can prove the following:

\begin{proposition}
The matrices $\bar{A}_{i}$, $i=0,1,2$, have the structure
\begin{subequations}
\begin{align}
& \bar{A}_{0}=\left[\begin{array}{ll}
\bar{A}_{011} & 0 \\ 
\bar{A}_{021} & 0
\end{array}\right],~~
\bar{A}_{011}=T_{R,2}A_0 W_{G,2},~~
\bar{A}_{021}=T_{R,1}A_0 W_{G,2},\label{19}\\
&\bar{A}_{1}=\left[\begin{array}{ll}
\bar{A}_{111} & \bar{A}_{112} \\
\bar{A}_{121} & \bar{A}_{122}
\end{array}\right],
\begin{array}{ll}
\bar{A}_{111}=T_{R,2}A_1W_{G,2},&
\bar{A}_{112}=T_{R,2}( A_1W_{G,1}+ A_2 W_{G,1} \Lambda_{G,1}), \\
  \bar{A}_{121}=(T_{R,1}A_1 +\Lambda_{R,1}T_{R,1}A_0)W_{G,2}, &
\bar{A}_{122}=T_{R,1}(A_1W_{G,1}+  A_2 W_{G,1} \Lambda_{G,1}),\end{array}\label{21}
\\
    & 
\bar{A}_{2}=\left[\begin{array}{ll}
\bar{A}_{211} & \bar{A}_{212} \\
 0 & 0
\end{array}\right],~~
\bar{A}_{211}=T_{R,2}A_2 W_{G,2},~~
\bar{A}_{021}=T_{R,2}A_2 W_{G,1}.\label{20}
\end{align}
\end{subequations}
\end{proposition}

The matrices $\bar{G}_{11}$ and $\bar{G}_{21}$, and $\bar{R}_{11}$ and $\bar{R}_{12}$, which form the matrices $\bar{G}$ and $\bar{R}$ in \eqref{17} and \eqref{18}, solve a system of matrix equations, as stated by the following:

\begin{theorem}
The matrices $\bar{G}_{11}$ and $\bar{G}_{21}$ of \eqref{17} satisfy
\begin{equation}\label{22}
    	\begin{cases}
    		\bar{A}_{021}+\bar{A}_{121}\bar{G}_{11}+\bar{A}_{122}\bar{G}_{21}=0,\\
    		\bar{A}_{011}+\bar{A}_{111}\bar{G}_{11}+\bar{A}_{112}\bar{G}_{21}+\bar{A}_{211}\bar{G}_{11}^2+\bar{A}_{212}\bar{G}_{21}\bar{G}_{11}=0.
    	\end{cases}
\end{equation}
The matrices 
$\bar{R}_{11}$ and $\bar{R}_{12}$ of \eqref{18} satisfy
\begin{equation}\label{23}
    	\begin{cases}
    		\bar{A}_{212}+\bar{R}_{11}\bar{A}_{112}+\bar{R}_{12}\bar{A}_{122}=0,\\
    		\bar{A}_{211}+\bar{R}_{11}\bar{A}_{111}+\bar{R}_{12}\bar{A}_{121}+\bar{R}_{11}^2\bar{A}_{011}+\bar{R}_{11}\bar{R}_{12}\bar{A}_{021}=0.
    	\end{cases}
\end{equation}    
\end{theorem}
\begin{proof}
Since $\bar{G}$ satisfies
\[
\bar{A}_{0}+\bar{A}_{1} \bar{G}+\bar{A}_{2} \bar{G}^2=0,
\]
from \eqref{17} and \eqref{19}-\eqref{20} we can obtain
\[
\left[\begin{array}{ll}
\bar{A}_{011} & 0 \\ 
\bar{A}_{021} & 0
\end{array}\right]+\left[\begin{array}{ll}
\bar{A}_{111} & \bar{A}_{112} \\
\bar{A}_{121} & \bar{A}_{122}
\end{array}\right] \left[\begin{array}{ll}
\bar{G}_{11}  & 0 \\
\bar{G}_{21} & 0
\end{array}\right]+\left[\begin{array}{ll}
\bar{A}_{211} & \bar{A}_{212} \\
 0 & 0
\end{array}\right] \left[\begin{array}{ll}
\bar{G}_{11}^2  & 0 \\
\bar{G}_{21}\bar{G}_{11} & 0
\end{array}\right]=0,
\]
which leads to \eqref{22}.

Similarly, since $\bar{R}$ satisfies
\[
\bar{A}_{2}+\bar{R} \bar{A}_{1}+\bar{R}^2 \bar{A}_{0}=0,
\]
from \eqref{18} and \eqref{19}-\eqref{20} we can obtain
\[
\left[\begin{array}{ll}
\bar{A}_{211} & \bar{A}_{212} \\
 0 & 0
\end{array}\right]+\left[\begin{array}{ll}
\bar{R}_{11} & \bar{R}_{12} \\
0 & 0
\end{array}\right] \left[\begin{array}{ll}
\bar{A}_{111} & \bar{A}_{112} \\
\bar{A}_{121} & \bar{A}_{122}
\end{array}\right]+\left[\begin{array}{ll}
\bar{R}_{11}^2 & \bar{R}_{11}\bar{R}_{12} \\
0 & 0
\end{array}\right]\left[\begin{array}{ll}
\bar{A}_{011} & 0 \\ 
\bar{A}_{021} & 0
\end{array}\right]=0,
\]
which leads to \eqref{23}.
\end{proof}

If the matrix $\bar{A}_{122} \in \mathbb{R}^{q\times q}$ is nonsingular, the systems \eqref{22} and \eqref{23} can be transformed into the QME
\begin{equation}\label{eq:qmeGbar}
B_0+B_1\bar{G}_{11}+B_2\bar{G}_{11}^2=0,
\end{equation}
 with
\begin{equation}\label{eq:Bbar}
    	\begin{cases}
            B_0=\bar{A}_{011}-\bar{A}_{112}\bar{A}_{122}^{-1}\bar{A}_{021},\\
    		B_1=\bar{A}_{111}-\bar{A}_{112}\bar{A}_{122}^{-1}\bar{A}_{121}-\bar{A}_{212}\bar{A}_{122}^{-1}\bar{A}_{021},\\
            B_2=\bar{A}_{211}-\bar{A}_{212}\bar{A}_{122}^{-1}\bar{A}_{121}.
    	\end{cases}
\end{equation}
and $\bar{G}_{21}$ can be recovered from the equation
\begin{equation}\label{G_21bar}
    \bar{G}_{21}=-\bar{A}_{122}^{-1}\left(\bar{A}_{021}+\bar{A}_{121}\bar{G}_{11}\right).
\end{equation}
In particular, since from \eqref{17} the eigenvalues of $\bar{G}_{11}$ are the eigenvalues of $\bar{G}$, then $\bar{G}_{11}$ is the solution of the QME \eqref{eq:qmeGbar}
with eigenvalues $\mu_1,\ldots,\mu_\ell$.

Similarly, we have
\begin{equation}\label{eq:qmeRbar}
\bar{R}_{11}^2B_0+\bar{R}_{11}B_1+B_2=0,
\end{equation}
and
\begin{equation}\label{R_12bar}
    	\bar{R}_{12}=-\left(\bar{A}_{212}+\bar{R}_{11}\bar{A}_{112}\right)\bar{A}_{122}^{-1},            
\end{equation}
where $\bar{R}_{11}$ is the solution to the QME \eqref{eq:qmeRbar} with eigenvalues $1/\mu_1,\ldots,1/\mu_\ell$.

The following result shows that the matrices $G$ and $R$ can be recovered from $\bar{G}_{11}$ and $\bar{G}_{21}$, and $\bar{R}_{11}$ and $\bar{R}_{12}$, respectively.

\begin{proposition}
\begin{equation}\label{G_final}
G=W_{G,2}\bar{G}_{11}W_{G,2}^\top+W_{G,1}\bar{G}_{21}W_{G,2}^\top+W_{G,1}\Lambda_{G,1}W_{G,1}^\top 
\end{equation}
and
\begin{equation}\label{R_final}
R=T_{R,2}^\top\bar{R}_{11}T_{R,2}+T_{R,2}^\top\bar{R}_{12}T_{R,1}+T_{R,1}^\top\Lambda_{R,1}T_{R,1}. 
\end{equation}    
\end{proposition}

\begin{proof}
From \eqref{eq:gbar} and \eqref{17}, we obtain
\[
    	\begin{split}
\widetilde{G}&=W_G\bar{G}W_G^\top\\
&=[W_{G,2} ~ W_{G,1}]\left[\begin{array}{ll}
\bar{G}_{11}  & 0 \\
\bar{G}_{21} & 0
\end{array}\right]\left[\begin{array}{ll}
W_{G,2}^\top \\
W_{G,1}^\top
\end{array}\right]\\
    		&=W_{G,2}\bar{G}_{11}W_{G,2}^\top+W_{G,1}\bar{G}_{21}W_{G,2}^\top,
    	\end{split}
\]
Subsequently, \eqref{G_final} is derived from
$G=\widetilde{G}+W_{G,1}{\Lambda}_{G,1} W_{G,1}^\top$.

Similarly, the expression for \eqref{R_final} can also be obtained.
\end{proof}

If $\Lambda_{G,1}$ is nonsingular, from \eqref{21} and the first formula in \eqref{10}, we can derive
\[
\bar{A}_{122}=-T_{R,1} A_0 W_{G,1}\Lambda_{G,1}^{-1},
\]
so that the condition for the non-singularity of the matrix $\bar{A}_{122}$ is
that the matrix
    		$T_{R,1} A_0 W_{G,1}$
            is nonsingular.
             It is worth noting that $\Lambda_{G,1}$ may be singular. In fact, $\Lambda_{G,1}$ is nonsingular if and only if $\lambda_1 \neq 0$.

\section{Convergence properties of CR in the case of more than one eigenvalue on the unit circle}\label{sec:convcr}
In this section, we derive some convergence properties of CR under the assumptions stated in Section~\ref{sec:ass}. In this case, Theorem~\ref{thm:crconv_orig} is not applicable, and CR generally does not converge if $\ell>1$.
However, we will show that CR converges on a suitable subspace, and such a subspace can be used to 
derive the matrices $W_{G,1}$, $\Lambda_{G,1}$, $T_{R,1}$, and $\Lambda_{R,1}$ used in Section~\ref{sec:defl}.

\begin{proposition}\label{prop:A0KW}
Let $W_{G,1}$ and $T_{R,1}$ be matrices such that  \eqref{G_invariant} and \eqref{R_invariant} hold.
Assume that CR can be carried out without a breakdown and that the sequences \eqref{eq:crgen} generated by CR are bounded. Then
\[
    \limsup_{k\to\infty} {\| A_0^{(k)}W_{G,1}\|^{2^{-k}}}\le |\lambda_{m-\ell}|, ~~~ \limsup_{k\to\infty} {\| T_{R,1} A_2^{(k)}\|^{2^{-k}}}\le 1/|\lambda_{m+\ell+1}|.
\]
Moreover, $\lim_{k\to\infty}\nullity(A_2^{(k)})=\lim_{k\to\infty}\nullity(A_0^{(k)})=m-\ell$, where  $\nullity(V)$ is the dimension of the null space of the matrix $V$.
\end{proposition}

\begin{proof}
By multiplying the first equation in \eqref{eq:crqme} to the right by $W_{G,1}$, since
\[
G^{2^k} W_{G,1}=W_{G,1}\Lambda_{G,1}^{2^k},
\]
we obtain
\[
A_{0}^{(k)} W_{G,1}+A_{1}^{(k)}W_{G,1}{\Lambda}_{G,1}^{2^k}+A_{2}^{(k)} W_{G,1}{\Lambda}_{G,1}^{2^{k+1}}=0.
\]
Since the sequences $\{ A_1^{(k)} \}_{k=1}^{\infty}$ and $\{ A_2^{(k)}\}_{k=1}^{\infty}$ are bounded and $\rho(\Lambda_{G,1})=|\lambda_{m-\ell}|<1$, then we conclude that
\[\limsup_{k\to\infty} {\| A_0^{(k)}W_{G,1}\|^{2^{-k}}}\le |\lambda_{m-\ell}|, \]
and
\[
\lim\limits _{k \rightarrow \infty}A_{0}^{(k)} W_{G,1}=0.
\]

Therefore, in particular $\limsup_{k\to\infty}{\nullity(A_0^{(k)})}\ge m-\ell$.
From the properties of the CR algorithm \cite{BM:CRsurvey}, the roots of 
\[a_k(z)=\det(A_0^{(k)}+z A_1^{(k)}+z^2 A_2^{(k)})\] are 
\[\lambda_i^{2^k}, \quad \text{for} \quad k=1,\ldots,2m.\]
From the  condition \eqref{eq:roots_separation}, as $k \rightarrow \infty$,  $a_k(z)$ has exactly $m-\ell$ zeros that converge to 0, which implies that $\lim_{k\to\infty}{\nullity(A_0^{(k)})}= m-\ell$.
The properties for the sequence $\{A_2^{(k)}\}_{k=1}^{\infty}$ can be proved similarly, by multiplying the second equation in \eqref{eq:crqme} to the left by $T_{R,1}$.
\end{proof}

Let $$A_{0}^{(k)}=U_{0}^{(k)}\Sigma_{0}^{(k)}(V_{0}^{(k)})^\top$$ be the SVD of $A_{0}^{(k)}$, where $U_0^{(k)}$
and $V_{0}^{(k)}$ are $m\times m$ orthogonal matrices, and
the diagonal entries of $\Sigma_{0}^{(k)}$ are in non-increasing order, partition
$V_0^{(k)}=[V_{0,1}^{(k)} ~ V_{0,2}^{(k)}]$ with  $V_{0,1}^{(k)}\in\mathbb{R}^{m\times \ell}$, $V_{0,2}^{(k)}\in\mathbb{R}^{m\times (m-\ell)}$.
Then, from Proposition \ref{prop:A0KW}, we can readily conclude that 
\[\operatorname{rank}(W_{G,1})=\lim\limits _{k \rightarrow \infty}\operatorname{rank}(V_{0,2}^{(k)}).\]
In particular, there exists a sequence of matrices $\{\widetilde{\Lambda}_{G,1}^{(k)}\}_{k=1}^{\infty}\in \mathbb{R}^{(m-l)\times (m-l)}$ 
such that
\begin{equation}\label{10-1}
\lim\limits _{k \rightarrow \infty}\left[GV_{0,2}^{(k)}-V_{0,2}^{(k)}\widetilde{\Lambda}_{G,1}^{(k)}\right]=0
\end{equation}
and the eigenvalues of $\widetilde{\Lambda}_{G,1}^{(k)}$, counting multiplicities, converge to the eigenvalues of ${\Lambda}_{G,1}$.
Therefore, for $k$ sufficiently large, the span of the columns of
$V_{0,2}^{(k)}$ approximates the invariant subspace of $G$ corresponding to the eigenvalues in the open unit disk, so that
$$
\lim\limits _{k \rightarrow \infty}\left[A_0V_{0,2}^{(k)}+A_1V_{0,2}^{(k)}\widetilde{\Lambda}_{G,1}^{(k)}+A_2V_{0,2}^{(k)}(\widetilde{\Lambda}_{G,1}^{(k)})^2\right]=0.
$$
In order to recover $\widetilde{\Lambda}_{G,1}^{(k)}$ for $k$ sufficiently large, multiply the last equation in \eqref{eq:crqme} by $V_{0,2}^{(k)}$ and obtain
\begin{equation}\label{eq:a0k}
A_{0}V_{0,2}^{(k)}+\widehat{A}_{1}^{(k)} G V_{0,2}^{(k)} + A_{2}^{(k)} G^{2^k+1}V_{0,2}^{(k)}=0.
\end{equation}
From \eqref{10-1}
and from the property $\lim_{k\to\infty}G^{2^k+1}V_{0,2}^{(k)}=\lim_{k\to\infty}V_{0,2}^{(k)}\left(\widetilde{\Lambda}_{G,1}^{(k)}\right)^{2^k+1}= 0$, 
by multiplying \eqref{eq:a0k} by $(V_{0,2}^{(k)})^\top$,
we deduce that
\[
\lim_{k\to\infty}\left(\widetilde{\Lambda}_{G,1}^{(k)} + (V_{0,2}^{(k)})^\top (\widehat{A}_{1}^{(k)})^{-1} A_0 V_{0,2}^{(k)}\right)=0,
\]
therefore $\widetilde{\Lambda}_{G,1}^{(k)}$ can be approximated by $-(V_{0,2}^{(k)})^\top (\widehat{A}_{1}^{(k)})^{-1} A_0 V_{0,2}^{(k)}$, for $k$ sufficiently large.
In other words, the invariant subspace corresponding to the eigenvalues in the open unit disk of the sought solution $G$ is approximated by the span generated by the columns of $V_{0,2}^{(k)}$. Moreover, a matrix $S_G$ such that
$G V_{0,2}^{(k)}\approx V_{0,2}^{(k)} S_{G}$ can be obtained as
\[
S_{G} = -(V_{0,2}^{(k)})^\top (\widehat{A}_{1}^{(k)})^{-1} A_0 V_{0,2}^{(k)}.
\]
Similarly, if
$$A_{2}^{(k)}=U_{2}^{(k)}\Sigma_{2}^{(k)}(V_{2}^{(k)})^\top$$ is the SVD of $A_{2}^{(k)}$, then partition $U_{2}^{(k)}$ as 
$U_{2}^{(k)}=\begin{bmatrix}
 U_{2,1}^{(k)} \\ U_{2,2}^{(k)}
\end{bmatrix}$, 
where $U_{2,1}^{(k)}\in\mathbb{R}^{\ell\times m}$, 
$U_{2,2}^{(k)}\in\mathbb{R}^{(m-\ell)\times m}$.
As for the matrix $G$, there exists a sequence of matrices $\{\widetilde{\Lambda}_{R,1}^{(k)}\}_{k=1}^{\infty}\in \mathbb{R}^{(m-l)\times (m-l)}$ 
such that
\[
\lim\limits _{k \rightarrow \infty}\left[U_{2,2}^{(k)} R-\widetilde{\Lambda}_{R,1}^{(k)} U_{2,2}^{(k)} \right]=0
\]
and the eigenvalues of $\widetilde{\Lambda}_{R,1}^{(k)}$, counting multiplicities, converge to the eigenvalues of ${\Lambda}_{R,1}$.
Therefore, the left invariant subspace corresponding to the eigenvalues in the open unit disk of the sought solution $R$ is approximated by the span generated by the rows of $U_{2,2}^{(k)}$. Moreover, a matrix $S_R$ such that
$ U_{2,2}^{(k)} R\approx S_R U_{2,2}^{(k)} $ can be obtained as
\[
S_{R} = -U_{2,2}^{(k)} A_2(\widehat{A}_{1}^{(k)})^{-1}  (U_{2,2}^{(k)})^\top.
\]

In particular, CR can be applied to compute left and right invariant subspace of $G$ and $R$, respectively, associated with the eigenvalues in the open unit disk.
The resulting procedure  is described by Algorithm~\ref{alg:isg}.

\begin{algorithm}
	\caption{Compute the right invariant subspace of $G$ and the left invariant subspace of $R$, associated with the eigenvalues in the open unit disk.}
	\label{alg:isg}
	\begin{algorithmic}
        \State \textbf{Input:} The $m\times m$ matrix coefficients $A_0$, $A_1$, and $A_2$ of the QME \eqref{eq:qme}; the number $\ell$, $1\le\ell\le m$, of eigenvalues of modulus 1 of the sought solutions $G$ and $R$; a tolerance $\epsilon>0$.
        \State \textbf{Output:} Orthogonal  matrices $W_G$ and $T_R$, and 
        $(m-\ell)\times (m-\ell)$ matrices $\Lambda_{G,1}$ and $\Lambda_{R,1}$, such that 
        $GW_{G,1}= W_{G,1} \Lambda_{G,1}$ and $T_{R,1} R= \Lambda_{R,1} T_{R,1} $, with
        $W_{G,1}\in \mathbb{R}^{m\times (m-\ell)}$ formed by the last $m-\ell$ columns of $W_G$,  $T_{R,1}\in \mathbb{R}^{(m-\ell)\times m}$ formed by the last $m-\ell$ rows of $T_R$, and $\sigma(\Lambda_{G,1})=\{\lambda_1,\ldots,\lambda_{m-\ell}\}$, $\sigma(\Lambda_{R,1})=\{1/\lambda_{m+\ell},\ldots,1/\lambda_{2m}\}$.
        \State \textbf{Computation:}
        \State Set $\widehat{A}_{1}^{(0)}=A_{1}, A_i^{(0)}=A_i$, for $i=0,1,2$.
        \For{$k=0,1,\ldots$}
        \State Compute  $\widehat{A}^{(k+1)}$, ${A_i^{(k+1)}}$, for $i=0,1,2$, by means of \eqref{eq:crgen}.
        \State Compute the SVD of $A_{0}^{(k+1)}$ and $A_{2}^{(k+1)}$, $A_{0}^{(k+1)}=U_{0}^{(k+1)}\Sigma_{0}^{(k+1)}(V_{0}^{(k+1)})^\top$ and {$A_{2}^{(k+1)}=U_{2}^{(k+1)}\Sigma_{2}^{(k+1)}(V_{2}^{(k+1)})^\top$}, respectively, and  let $\sigma_1^{(j)}\ge \cdots \ge \sigma_m^{(j)}$, $j=0,2$, be the corresponding singular values.
        \If{$\sigma^{(j)}_{\ell+1}/\sigma^{(j)}_{\ell}<\epsilon$, for $j=0,2$,}
        \State\Return $W_{G}=V_{0}^{(k+1)}$, $T_{R}=U_{2}^{(k+1)}$,  $\Lambda_{G,1}=-W_{G,1}^\top (\widehat{A}_{1}^{(k+1)})^{-1} A_0 W_{G,1}$, $\Lambda_{R,1} = -T_{R,1} A_2 (\widehat{A}_{1}^{(k)})^{-1}  T_{R,1}^\top$.
        \EndIf
        \EndFor
	\end{algorithmic}
\end{algorithm}

\section{Combination of CR, block shift-and-deflate, and QZ algorithm}\label{CR and block shift-and-deflate}
In this section, we use the properties shown in Sections~\ref{sec:defl} and \ref{sec:convcr} to derive an algorithm for computing the solution $G$, by separating the invariant subspace corresponding to the  
eigenvalues inside and on the unit circle. 
Once $G$ is computed, we may recover $R$ from \eqref{eq:GRrelation}.

More specifically, once the matrices $W_G$, $\Lambda_G$, $T_R$, and $\Lambda_R$ are computed by means of Algorithm~\ref{alg:isg}, 
we construct the matrices $\bar{A}_i$, $i=0,1,2$, be means of formulas \eqref{19}, \eqref{21}, and \eqref{20}. Then we solve the $\ell\times\ell$ quadratic matrix equations \eqref{eq:qmeGbar} and \eqref{eq:qmeRbar}. These latter equations cannot be solved by CR, since the corresponding matrix polynomials have all the eigenvalues of modulus one and CR does not generally converge.

An alternative approach  can be formulated in terms of the eigensystems of an associated generalized eigenvalue problem. Higham and Kim established the following lemma \cite{Higham2000}, which characterizes the solutions of 
\begin{equation}\label{eq:qmeB}
B_0+B_1X+B_2 X^2=0.    
\end{equation}

\begin{lemma}
Define

$$
M=\left[\begin{array}{cc}
0 & I \\
-B_0 & -B_1
\end{array}\right], \quad N=\left[\begin{array}{cc}
I & 0 \\
0 & B_2
\end{array}\right],
$$ 
then $X$ is a solution of \eqref{eq:qmeB} if and only if

$$
M\left[\begin{array}{l}
I \\
X
\end{array}\right]=N\left[\begin{array}{l}
I \\
X
\end{array}\right] X.
$$
\end{lemma}

Under our assumptions, the matrix polynomial $B(z)=B_0+B_1 z + B_2 z^2$ associated with \eqref{eq:qmeGbar} has eigenvalues $\mu_1,\ldots,\mu_\ell$, each with multiplicity 2, and $\bar{G}_{11}$ is the only solution to \eqref{eq:qmeGbar} with eigenvalues $\mu_1,\ldots,\mu_\ell$. To compute $\bar{G}_{11}$, by borrowing the idea from \cite{Higham2000}, we employ the generalized Schur decomposition, which can be computed via the QZ algorithm \cite[Chapter 7]{gvl}, and derive the following Algorithm~\ref{alg:QZ}. Similarly, $\bar{R}_{11}$ is the only solution to \eqref{eq:qmeRbar} with eigenvalues $1/\mu_1,\ldots,1/\mu_\ell$. Analogous to \eqref{eq:GRrelation}, $\bar{R}_{11}$ can be computed from
\begin{equation}\label{eq:barGRrelation}
    \bar{R}_{11} = -B_2(B_2\bar{G}_{11} + B_1)^{-1}. 
\end{equation}


\begin{algorithm}
	\caption{Compute the solution $\bar{G}_{11}$ of the QME \eqref{eq:qmeB} by means of the QZ algorithm}
	\label{alg:QZ}
	\begin{algorithmic}
        \State \textbf{Input:} the $\ell\times \ell$ matrix coefficients $B_0$, $B_1$, and $B_2$ of the QME \eqref{eq:qmeB}.
        \State \textbf{Output:} the solution $\bar{G}_{11}$ of the QME \eqref{eq:qmeB} with eigenvalues $\mu_1,\ldots,\mu_\ell$.
        \State \textbf{Computation:}
        \State Compute the generalized Schur decomposition 
        $$
Q^* M Z=T, \quad Q^* N Z=S
$$
with $Q$ and $Z$ unitary and $T$ and $S$ upper triangular, and where all matrices are partitioned as block $2 \times 2$ matrices with $\ell \times \ell$ blocks.
        \State Compute the solution $$\bar{G}_{11}=Z_{21} Z_{11}^{-1}=Q_{11} T_{11} S_{11}^{-1} Q_{11}^{-1}.$$
	\end{algorithmic}
\end{algorithm}

Following the above discussion, the resulting procedure for solving the QMEs \eqref{eq:qme} and \eqref{eq:rqme} is summarized in Algorithm~\ref{alg:gfull}. The proposed algorithm is termed the Block-Shifted Cyclic Reduction algorithm, abbreviated as \texttt{BS-CR}.

\begin{algorithm}
	\caption{Compute the solution $G$ and  $R$ of the QMEs \eqref{eq:qme} and \eqref{eq:rqme}, respectively, by means of the \texttt{BS-CR} algorithm}
	\label{alg:gfull}
	\begin{algorithmic}
        \State \textbf{Input:} The $m\times m$ matrix coefficients $A_0$, $A_1$, and $A_2$ of the QME \eqref{eq:qme}; the number $\ell$, $1\le\ell\le m$, of eigenvalues of modulus 1 of the sought solutions $G$ and $R$; a tolerance $\epsilon>0$.
        \State \textbf{Output:} the solution $G$ and  $R$ of the QME \eqref{eq:qme} and \eqref{eq:rqme}, respectively.
        \State \textbf{Computation:}
        \State Compute the matrices $W_G$, $\Lambda_{G,1}$, $T_R$, and $\Lambda_{R,1}$, by means of Algorithm~\ref{alg:isg}.
        \State Compute $B_i$, $i=0,1,2$, by means of \eqref{eq:Bbar}.
        \State Solve the $\ell\times\ell$ quadratic matrix equation \eqref{eq:qmeGbar} by means of the Algorithm~\ref{alg:QZ} and obtain $\bar{G}_{11}$.
        \State Compute $\bar{R}_{11}$ by applying \eqref{eq:barGRrelation}.
        \If{$\bar{A}_{122}$ in \eqref{21} is nonsingular}
        \State Compute $\bar{G}_{21}$ and $\bar{R}_{12}$ 
        by applying \eqref{G_21bar} and \eqref{R_12bar}, respectively.
        \EndIf
        \State Compute $G$ and $R$, 
        by means of \eqref{G_final} and \eqref{R_final}, respectively.
	\end{algorithmic}
\end{algorithm}

\section{Numerical results}\label{Numerical results}
We test the proposed algorithm on two null recurrent QBD problems (Examples 1 and 2), where the number $\ell$ of double eigenvalues of $A(z)$ on the unit circle is greater than 1. 
In particular, the coefficients $A_0$, $A_1$, and $A_2$ are
\[
A_0=-E_0,~~A_1=I-E_1,~~A_2=-E_2,
\]
where $E_i \ge 0$, for $i=0,1,2$, and $(E_0+E_1+E_2)\mathbf{1}=\mathbf{1}$. Moreover, we test the proposed algorithm on an artificial example, not coming from stochastic processes, to verify its robustness (Example 3).

We compare the proposed Algorithm~\ref{alg:gfull}, denoted by \texttt{BS-CR}, with standard CR, denoted by \texttt{CR} and described in \cite[Section 7.5]{BLM05}, the Shifted CR, denoted by \texttt{S-CR} and proposed in \cite{HMR}, and the $U$-based fixed point iteration \texttt{FPI} with starting approximation a stochastic matrix (see \cite[Chapter 6]{BLM05}).

The algorithm \texttt{BS-CR} has been implemented in Matlab.
For \texttt{CR}, \texttt{S-CR}, and \texttt{FPI}, we used the Matlab implementation provided in the package SMCSolver \cite{smcsolver}.

The experiments have been performed with Matlab R2023a, on a Laptop Intel(R) Core(TM) i7  3.00GHz.
For each algorithm, we report the iteration counts, the CPU time and the residual error
defined as 
$$\textbf{Residual}=\|A_0+(A_1 + A_2 X)X \|_\infty,$$ 
where $X$ is the computed approximation of $G$. The residual for the computed approximation for $R$ is not reported since in all the experiments it is almost the same as the residual for the approximation of $G$. 

\subsection{Example 1}
This test is a special case of the example in Section 4.4 of the book \cite{BLM05}. The coefficients $E_i$, for $i=0,1,2$, are
\[
E_0=\begin{bmatrix}
     0 & 0 & 0 & \frac14\\
   \frac{33}{160} & 0 & 0 &0  \\
   \frac{1}{4} & 0 & 0 & 0 \\
   0 & \frac14 & 0& 0 \\
\end{bmatrix},~~
E_1=\begin{bmatrix}
    0 & 0 & 0 & 0 \\
    0 & 0 & \frac34 & 0 \\
    0 & \frac34 & 0 & 0\\
    0 & 0 & 0 & 0 \\
\end{bmatrix},~~E_2=\begin{bmatrix}
     0 & \frac34 & 0 & 0 \\
     0 & 0 & 0 & \frac{7}{160} \\
   0 & 0 & 0 & 0 \\
   \frac34 & 0 & 0 & 0 \\
\end{bmatrix}.
\]
The eigenvalues of $A(z)$ are $\lambda_1=0$, $\mu_j=\cos \frac{2\pi(j-1)}{3}+ i \sin \frac{2\pi(j-1)}{3} $, for $j=1,2,3$, each with multiplicity 2, and $\lambda_8=\infty$. Therefore, this example corresponds to a null recurrent QBD, with $\ell=3$ double eigenvalues on the unit circle.

Table~\ref{tab:example1} reports the obtained results. The fixed point iteration exhibits sublinear convergence, and 200,000 iterations are needed to reach a residual of order $10^{-10}$.
CR and S-CR show a linear convergence. In particular, due to the presence of $\ell>1$ eigenvalues of $A(z)$ on the unit circle, the standard shift technique is unable to restore quadratic convergence of CR. Instead, the proposed block shift technique leads to the convergence of CR in just one step. Concerning residual error, \texttt{BS-CR} is clearly the most accurate. In this example the CPU time is not meaningful since the size of the blocks is small.

\begin{table}[h]
    \centering
    \caption{Numerical results for Example 1}
    \label{tab:example1}
    \begin{tabular}{|c|c|c|c|c|}
        \hline
         & \textbf{Iterations} & \textbf{CPU time} & \textbf{Residual}   \\
        \hline
        \texttt{BS-CR} & 1 & 0.01 & $3.9 \cdot 10^{-15}$ \\
        \hline
        \texttt{S-CR} & 29 & 0.02 & $3.0\cdot 10^{-12}$  \\
        \hline
        \texttt{CR} & 30 & 0.02 & $4.4\cdot 10^{-16}$  \\
        \hline
        \texttt{FPI} & 200,000 & 0.5 & $1.5\cdot 10^{-10}$  \\
        \hline
    \end{tabular}
\end{table}

\subsection{Example 2}
In this example
\[
E_0=\begin{bmatrix}
    0 & F_1\\
    F_2 & 0
\end{bmatrix},~~E_1=0,~~E_2=\begin{bmatrix}
    0 & F_2\\
    F_1 & 0
\end{bmatrix},
\]
where
 $F_1\ge 0$ and $F_2\ge 0$ are $p\times p$ matrices such that $(F_1+F_2)\mathbf{1}=\mathbf{1}$. This example generalizes the example on page 133 in \cite{ght96}.
 Here we take
 \[
 F_1=\frac18 \begin{bmatrix}
     3 & 1\\
     1 & 2 & 1\\
     & \ddots & \ddots & \ddots \\
     && 1 & 2 & 1 \\
     &&& 1 &3
 \end{bmatrix},~~ 
 F_2=\frac{1}{10} \begin{bmatrix}
     4 & 1\\
     1 & 3 & 1\\
     & \ddots & \ddots & \ddots \\
     && 1 & 3 & 1 \\
     &&& 1 &4
 \end{bmatrix}.
 \]
 The matrix polynomial $A(z)$ has $\ell=2$ double eigenvalues at the square root of $1$.
 We test the robustness of different versions of the CR algorithm (BS-CR, S-CR, and CR) for increasing values of the size $p$. In our implementations, the iterations are terminated once the maximum number of iterations is reached, with $k_{\max}=100$ for \texttt{CR} and \texttt{S-CR}, and $k_{\max}=12$ for \texttt{BS-CR}. The numerical results are presented in Figure~\ref{fig:ex2}, which displays the residual errors for the three algorithms. As observed from the figure, \texttt{S-CR} for some case provides an error of the order $10^0$, \texttt{CR}  produces an approximation of the order of the square root of the machine precision, while \texttt{BS-CR} achieves the highest accuracy, of the order $10^{-15}$. 
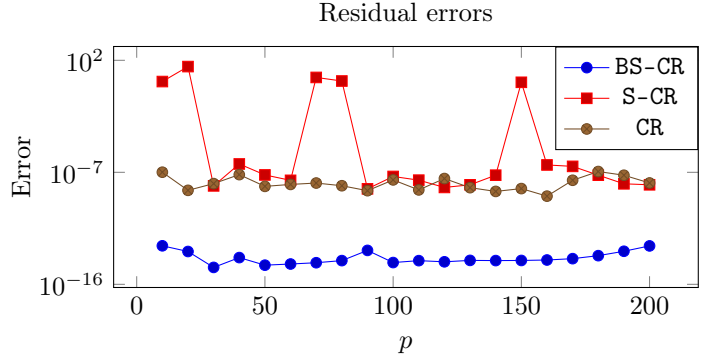
\begin{figure}[htbp]
\begin{center}
 \begin{tikzpicture}
        \begin{semilogyaxis}[
                legend style={at={(1,1)},anchor=north east},
                legend columns=1,
                width = .6\linewidth, height = .2\textheight,
                xlabel = {$p$}, 
                ylabel = {Error}, title = {Residual errors}]
            \addplot table[x index = 0, y index = 1] {example2.txt};
            \addplot table[x index = 0, y index = 2] {example2.txt};
            \addplot table[x index = 0, y index = 3] {example2.txt};
            \legend{\texttt{BS-CR},             \texttt{S-CR},
               \texttt{CR}
              }
        \end{semilogyaxis}
\end{tikzpicture}
\end{center}
\caption{Residual errors for the different versions of CR applied to Example 2}\label{fig:ex2}
\end{figure}

In Figure~\ref{fig:ex2time} we report the CPU time and the number of iterations needed by the different versions of CR. As we expect, \texttt{S-CR} and \texttt{CR} require a lower CPU time with respect to \texttt{BS-CR}. However, the asymptotic growth of the three algorithms almost the same. Indeed, the larger computational cost of each step of \texttt{BS-CR} is balanced by the lower number of iterations; moreover \texttt{CR} always reached the maximum number of iteration and, even performing more iterations, the residual does not improve.

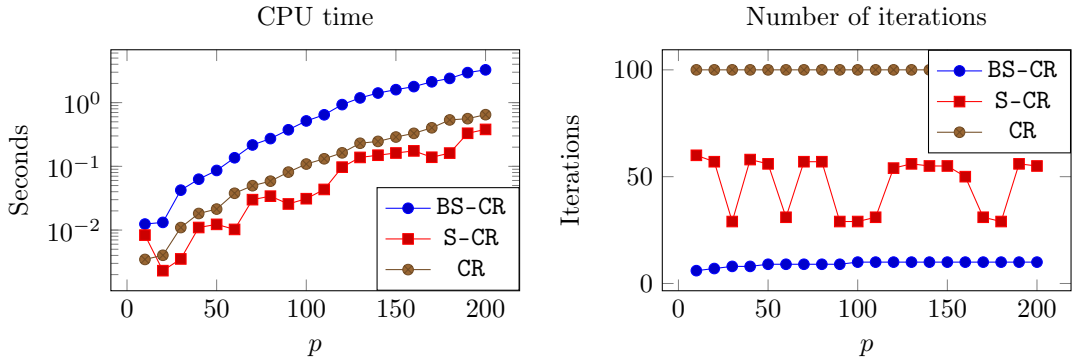
\begin{figure}[htbp]
\begin{center}
\begin{tabular}{cc}

\begin{tikzpicture}
        \begin{semilogyaxis}[
                legend style={at={(1,0)},anchor=south east},
                legend columns=1,
                width = .45\linewidth, height = .2\textheight,
                xlabel = {$p$}, 
                ylabel = {Seconds}, title = {CPU time}]
            \addplot table[x index = 0, y index = 1] {example2time.txt};
            \addplot table[x index = 0, y index = 2] {example2time.txt};
            \addplot table[x index = 0, y index = 3] {example2time.txt};
            \legend{\texttt{BS-CR},             \texttt{S-CR},
               \texttt{CR}
              }
        \end{semilogyaxis}
        \end{tikzpicture}

&

\begin{tikzpicture}
        \begin{axis}[
                legend style={at={(1,1)},anchor=north east},
                legend columns=1,
                width = .45\linewidth, height = .2\textheight,
                xlabel = {$p$}, 
                ylabel = {Iterations}, title = {Number of iterations}]
            \addplot table[x index = 0, y index = 1] {example2iter.txt};
            \addplot table[x index = 0, y index = 2] {example2iter.txt};
            \addplot table[x index = 0, y index = 3] {example2iter.txt};
            \legend{\texttt{BS-CR},             \texttt{S-CR},
               \texttt{CR}
              }
        \end{axis}
        \end{tikzpicture}

\end{tabular}
\end{center} 
\caption{CPU time and number of iterations for the different versions of CR applied to Example 2}\label{fig:ex2time}
\end{figure}

\subsection{Example 3}
The $m\times m$ coefficients of the quadratic matrix polynomial $A(z)=A_{0}+A_{1} z+A_{2} z^2$
are obtained by the coefficients of the same degree in the formula
$$A(z)=(zR-I)H(zI-G),$$
namely
$$A_{0}=HG,~~ A_{1}=-RHG-P \quad\text{and}\quad A_{2}=RH,$$
where $H$ is a nonsingular matrix. By construction, $G$ and $R$ are solutions of 
 the QMEs \eqref{eq:qme} and \eqref{eq:rqme}, respectively.
We choose
$$H = \operatorname{tridiag}(-1, 4, -1) \in \mathbb{R}^{m \times m},$$ 

\begin{equation*}
G=\left[\begin{array}{ll}
G_{11}  & G_{12} \\
0 & G_{22}
\end{array}\right],~~
R=\left[\begin{array}{ll}
R_{11}  & R_{12} \\
0 & R_{22}
\end{array}\right], 
\end{equation*}
where
\begin{equation*}
\begin{split}
&G_{11}=\operatorname{diag}\left(\mu_1, \ldots, \mu_\ell\right),~~
G_{12}=\operatorname{rand}\left(\ell, m-\ell\right),~~
G_{22}=\operatorname{diag}\left(\lambda_{1}, \ldots, \lambda_{m-\ell}\right),\\&
R_{11}=G_{11}^{-1},~~
R_{12}=\operatorname{rand}\left(\ell, m-\ell\right),~~
R_{22}=\frac{2}{3}G_{22},
\end{split}    
\end{equation*}
with $\lambda_k=\frac{1}{3}+\frac{1}{\ell+k}$, $k=1,\ldots,m-\ell$, and $\mu_i$, $i=1,\ldots,\ell$, have modulus 1 and are defined in Table~\ref{tab:example3mu}. Here, $\operatorname{rand}(h,k)$ is a $h\times k$ random matrix generated by the Matlab command $\texttt{rand}$.

\begin{table}[h]
    \centering
    \caption{Test cases for Example 3}
    \label{tab:example3mu}
    \begin{tabular}{|c|c|c|}
        \hline
    \textbf{Case}     & $\ell$ & $\mu_1,\ldots,\mu_\ell$   \\
        \hline
        1 & 2 & $0.6+0.8\mathrm{i}$, $-1$ \\
        2 & 4 & $0.6+0.8\mathrm{i}$, $1$, $-0.8-0.6\mathrm{i}$, $-1$ \\
        3 & 8 & $0.6+0.8\mathrm{i}$, $1$, $-0.8-0.6\mathrm{i}$, $-1$, $-0.6+0.8\mathrm{i}$, $1$, $0.6-0.8\mathrm{i}$, $-1$\\
     \hline
    \end{tabular}
\end{table}

One can easily verify that both $G$ and $R$ have $\ell$ eigenvalues on the unit circle and $m-\ell$ eigenvalues inside the unit circle. Moreover, $\mu_1,\ldots,\mu_\ell$ are eigenvalues of $A(z)$ of multiplicity 2.
Differently from 
the assumptions stated in Section~\ref{sec:ass}, both $G$ and $R$ may have complex entries, as well as the matrix coefficients of $A(z)$. However, the convergence theory developed in the paper can be easily extended to the complex case.

Since the implementation of the \texttt{S-CR} algorithm requires the eigenvalues of modulus $1$ and their corresponding eigenvectors to be known, we only test here the robustness of \texttt{BS-CR} and \texttt{CR} for different values of $m$ and $\ell$. 

In our implementations, the iteration is terminated once the current iterate of both algorithms satisfies $\textbf{Residual} \le 10^{-7}$, or when the maximum prescribed number of iterations, $k_{\max} = 100$, is exceeded. The latter case is labeled as “–” in the numerical tables.

The numerical results are presented in Table~\ref{tab:example3rev}, which reports the iteration counts and residual errors for the two algorithms.
  We observe that the algorithm \texttt{BS-CR} converges in fewer iterations, regardless of the tested values of $m$, and is significantly more accurate than \texttt{CR}.  The accuracy of \texttt{BS-CR} slightly worsens as $m$ increases.

\begin{table}[!h]
\caption{Numerical results of \texttt{BS-CR} and \texttt{CR} for Example 3}\label{tab:example3rev}
\begin{center}
\begin{tabular}{p{1.8cm}p{1.8cm}p{1.8cm}p{1.8cm}p{1.8cm}p{1.8cm}lp{1.8cm}p{1.8cm}p{1.8cm}}
\hline\noalign{\smallskip}

Method&&\multicolumn{2}{l}{\texttt{BS-CR}}&&\multicolumn{2}{l}{\texttt{CR}}\\

\cline{3-4}\cline{6-7}\noalign{\smallskip}

 & \textbf{Case} &\textbf{Iterations}&\textbf{Residual}&&\textbf{Iterations}&\textbf{Residual}\\

\hline\noalign{\smallskip}

  $m=16$              &$1$      &4    &1.23e-12       &&17      &6.11e-09\\
                      &$2$      &4    &8.44e-13       &&19      &5.52e-09\\
                      &$3$      &4    &1.52e-12       &&-       &3.03e-06\\ \hline
  $m=32$              &$1$      &4    &2.27e-12       &&18      &3.56e-09\\
                      &$2$      &4    &3.84e-12       &&18      &7.63e-09\\
                      &$3$      &4    &1.06e-11       &&-       &4.94e-06\\ \hline
  $m=64$              &$1$      &4    &7.49e-11       &&17      &7.19e-09\\
                      &$2$      &4    &6.58e-10       &&-       &1.38e-06\\
                      &$3$      &4    &5.90e-10       &&-       &8.39e-06\\ \hline
  $m=128$             &$1$      &4    &5.49e-11       &&19      &3.05e-09\\
                      &$2$      &4    &5.36e-10       &&-       &3.86e-06\\
                      &$3$      &4    &1.91e-10       &&-       &1.24e-05\\
\noalign{\smallskip}\hline
\end{tabular}
\end{center}
\end{table}

\section{Conclusions}\label{Conclusions}
For the unilateral quadratic matrix equations and their reversed forms, the convergence of the CR algorithm and its shift variant cannot be guaranteed when the associated polynomial has more than one eigenvalue on the unit circle. To overcome this limitation, we propose a novel iterative method, referred to as the Block-Shifted CR (\texttt{BS-CR}) algorithm, that enhances CR by combining SVD and block shift-and-deflate techniques. The resulting approach extends the applicability of existing solvers to a broader class of quadratic matrix equations. Numerical experiments confirm the effectiveness and robustness of the proposed method.





\section*{Acknowledgements}
The authors appreciate the constructive suggestions provided by the referees, which have significantly enhanced the clarity and quality of the paper.
This work was initiated and primarily conducted during the first author's visit to the University of Pisa. The first author is deeply grateful to Prof. Beatrice Meini for her kind invitation and hospitality, and to the Department of Mathematics for providing the excellent research environment that made this collaboration possible.

\section*{Declarations}


\begin{itemize}
\item Funding

This research was partly funded by: the China Scholarship Council (No. 202208625004) and the National Natural Science Foundation of China (No. 11501272);  the Italian Ministry of University and Research (MUR) through the PRIN 2022 ``Low-rank Structures and Numerical Methods in Matrix and Tensor Computations and their Application'' code: 20227PCCKZ MUR D.D. financing decree n. 104 of February 2nd, 2022 (CUP I53D23002280006); the MUR Excellence Department Project awarded to the Department of Mathematics, University of Pisa, CUP I57G22000700001. The second author is a member of GNCS-INdAM. 

\item Conflict of interest/Competing interests

The authors declare no competing interests.
\item Ethics approval and consent to participate

Not applicable.
\item Author contribution

The contributions of the authors are equal.
\end{itemize}


\bibliographystyle{elsarticle-num-names}
\bibliography{biblio}

\end{document}